\documentclass[12pt]{amsart}
\usepackage{amsfonts, amsbsy, amsmath, amssymb}

\usepackage[T1]{fontenc}
\usepackage[utf8]{inputenc}
\usepackage{lmodern}
\usepackage{amsmath, amsthm, amssymb,amscd, mathrsfs, amsfonts, mathtools,tikz-cd}
\usepackage{relsize}
\usepackage{euler}
\usepackage{pbox}
\usepackage{times}
\usepackage[all]{xy}
\usepackage{todonotes}
\usepackage{xcolor}
\usepackage[bookmarks=true,
bookmarksnumbered=true, breaklinks=true,
pdfstartview=FitH, hyperfigures=false,
plainpages=false, naturalnames=true,
colorlinks=true,pagebackref=true,
pdfpagelabels]{hyperref}
\usepackage{tikz}
\usepackage{pgfplots}

\pgfplotsset{compat=1.10}
\usepgfplotslibrary{fillbetween}
\usetikzlibrary{patterns}
\usetikzlibrary{matrix}

\usepackage[lite,alphabetic]{amsrefs}

\renewcommand{\PrintDOI}[1]{\href{http://dx.doi.org/\detokenize{#1}}{doi: \detokenize{#1}}%
	\IfEmptyBibField{pages}{, (to appear in print)}{}}

\def\commutatif{\ar@{}[rd]|{\circlearrowleft}}

\newtheorem{thm}{Theorem}[section]
\newtheorem{prop}[thm]{Proposition}

\theoremstyle{definition}
\newtheorem{defn}[thm]{Definition}

\theoremstyle{remark}
\newtheorem{rmk}[thm]{Remark}

\newtheorem{ex}[thm]{Example}

\allowdisplaybreaks

\usepackage{bbm}

\def\a{\alpha}
\def\b{\beta}

\def\t{\theta}
\def\s{\sigma}

\hypersetup{
	colorlinks = true,
	urlcolor = blue,
	linkcolor = blue,
	citecolor = red,
	pdfpagemode = UseNone
}

\title[Dividends of the BKR equivalence]{Dividends of the BKR Equivalence:
Quantifying the Dichotomy for Skyscraper Sheaves}

\author{Boris Tsvelikhovskiy} 
\address{Department of Mathematics, University of California, Riverside, CA 92521, USA} 
\email{borist@ucr.edu}
\begin{document}

\begin{abstract}
Let $G\subset \mathrm{SL}_3(\mathbb C)$ be a nontrivial finite abelian
subgroup and let $\Psi \colon D^b_G(\mathbb{C}^3) \xrightarrow{\sim} D^b(G\operatorname{-}\mathrm{Hilb}(\mathbb{C}^3))$ denote the Bridgeland--King--Reid
equivalence. For every nontrivial character $\chi$ of $G$,
the transform of the corresponding skyscraper sheaf at the
origin has exactly one nonzero cohomology sheaf, in degree
$0$ or $-1$. We study the proportion $\mathcal B_0$ of
characters for which this degree is zero. Using derived
Reid's recipe, we establish the sharp bounds
$1/3\leq\mathcal B_0\leq1$.
For isolated quotient singularities, the lower bound sharpens to $1/2$, and we exhibit an explicit family of cyclic quotients indexed by positive integers $m$ for which $\mathcal{B}_0 \to 1$ as $m \to \infty$. 
We show that $\mathcal B_0=1$ precisely when the fiber
over the origin has dimension at most one, and relate $\mathcal B_0$ to the Picard rank of $G\operatorname{-}\mathrm{Hilb}(\mathbb{C}^3)$ and the number of compact exceptional divisors of the resolution $G\operatorname{-}\mathrm{Hilb}(\mathbb{C}^3)\rightarrow \mathbb{C}^3/G$. A theorem of Craw--Lin--Yamagishi translates these estimates into bounds for the coordinate facets of the $G$-Hilbert GIT chamber. We discuss how the proportion $\mathcal{B}_0$ adapts to non-abelian finite subgroups of $\mathrm{SL}_3(\mathbb{C})$.
\end{abstract}

\maketitle
\section{Introduction}
The classical McKay correspondence concerns a finite
subgroup $G\subset \mathrm{SL}_2(\mathbb C)$ and the minimal resolution $S\to\mathbb C^2/G$. It gives a bijection between the nontrivial irreducible representations of $G$ and the irreducible exceptional curves of $S$. In higher dimensions, the correspondence takes several forms. For a finite subgroup $G \subset \mathrm{SL}_n(\mathbb{C})$, write $X = \mathbb{C}^n/G$ for the affine quotient variety and let $Y \to X$ be a crepant resolution.
The derived McKay correspondence relates coherent sheaves
on $Y$ to $G$-equivariant coherent sheaves on $\mathbb C^n$.
 In modern language it is usually understood as an equivalence of triangulated categories $D_G^b(\mathrm{Coh}(\mathbb{C}^n))$ and $D^b(\mathrm{Coh}(Y))$. Such an equivalence is known to hold in the following cases: 
\begin{itemize}
	\item $G\subset \mathrm{SL}_2{(\mathbb{C})}$, any $G$ and
the minimal resolution $Y\to\mathbb C^2/G$  (\cite{KV});
	\item $G\subset \mathrm{SL}_3{(\mathbb{C})}$, any $G$, $Y=G\operatorname{-}\mathrm{Hilb}(\mathbb{C}^3)$ (\cite{BKR});
	\item $G\subset \mathrm{SL}_3{(\mathbb{C})}$, any abelian $G$ and any projective crepant resolution $(Y,\rho)$ (\cite{CI});
	\item  $G\subset \operatorname{Sp}_{2n}(\mathbb{C})$, any crepant symplectic resolution  $(Y,\rho)$ (\cite{BK});
	\item $G\subset \mathrm{SL}_n{(\mathbb{C})}$, any abelian $G$,  any crepant symplectic resolution $(Y,\rho)$ (\cite{KawT}).
\end{itemize}

Any finite-dimensional representation $V$ of $G$ gives rise to two equivariant sheaves on
$\mathbb{C}^n$: 
\begin{itemize}
	\item the skyscraper sheaf $V^!=V\otimes_{\mathbb{C}} \mathcal{O}_{0}$, whose fiber at $0$ is $V$ and all the other fibers vanish;
	\item  the locally free sheaf $\widetilde{V}=V\otimes_{\mathbb{C}} \mathcal{O}_{\mathbb{C}^n}$.
\end{itemize}

Unless stated otherwise, we fix $G \subset \mathrm{SL}_3(\mathbb{C})$ to be a non-trivial finite abelian subgroup of order $r = \vert{}G\vert{}$. Throughout, $\Psi$ denotes the inverse of the standard Bridgeland--King--Reid equivalence induced by the universal family; its precise normalization is recalled in Section~\ref{subsec:BKR_reminder}.

A natural question arises in this context.  

\medskip
	\emph{What are the images of $\widetilde{\chi}$ and $\chi^!$ (for nontrivial irreducible representations $\chi$ of $G$) under the equivalence $\Psi$?}
\medskip

The first case is well understood: the object $\Psi(\widetilde{\chi})$ is a vector bundle of rank $\dim(\chi)$, the dual of a \emph{tautological bundle} or \emph{GSp--V sheaf} (after Gonzalez--Sprinberg and Verdier, see \cite{GSV}).  
Much less is known about $\Psi(\chi^!)$. In dimension two, and for abelian groups in dimension
three, the objects $\Psi(\chi^!)$ admit explicit geometric descriptions, due respectively to Kapranov--Vasserot \cite{KV} and Cautis--Craw--Logvinenko \cite{CCL}.
We recall the following consequences of these
descriptions: 

\begin{enumerate}
\item For a finite subgroup $G\subset \mathrm{SL}_2(\mathbb C)$
and a nontrivial irreducible representation $\chi$, the equivalence sends the corresponding simple sheaf to
\[
\Psi(\mathcal O_0\otimes\chi)
\simeq\mathcal O_{C_\chi}(-1)
\]
for a uniquely determined irreducible exceptional curve
$C_\chi$. The assignment
\[
\rho\longmapsto
\operatorname{Supp}\Psi(\mathcal O_0\otimes\chi)=C_\chi
\]
is a bijection between the nontrivial irreducible
representations and the irreducible exceptional curves.

	\item If $G\subset \mathrm{SL}_3(\mathbb{C})$ is a finite abelian subgroup, then for any nontrivial character $\chi$, the object $\Psi(\chi^!) \in D^b(\mathrm{Coh}(Y))$ is \emph{pure} and concentrated in degree $0$ or $-1$.  
	Here $Y=G\!\operatorname{-}\mathrm{Hilb}(\mathbb{C}^3)$, and an object is called \textbf{pure} if all of its cohomology sheaves vanish except in a single degree.
\end{enumerate}

For $E\in D^b(\operatorname{Coh}(Y))$, we write
$H^i(E)$ for its $i$-th cohomology sheaf.
The preceding dichotomy motivates the following definition.

\begin{defn}
    Let $\mathfrak{H}_0$ denote the subset of nontrivial characters $\chi \in G^\vee \setminus \{1\}$ for which the zeroth cohomology of the complex $\Psi(\chi^!)$ is non-vanishing, so that $H^0\bigl(\Psi(\chi^!)\bigr) \neq 0$. 
We define the proportion
\[
\mathcal{B}_0 \coloneqq \frac{|\mathfrak{H}_0|}{r-1},
\]
which measures the exact fraction of nontrivial characters whose transforms under $\Psi$ are cohomologically concentrated in degree zero.
    \label{ShareDefn}
\end{defn}

In this paper we investigate the possible values of
$\mathcal B_0$. Theorem~\ref{MainThm} establishes the sharp
universal bounds
\[
\frac13\leq\mathcal B_0\leq1.
\]
The lower bound is attained by the cyclic group action of type
$\frac14(1,1,2)$, and the upper bound by the diagonal
Klein four group. The proof yields a refined counting
formula in terms of the vertex and edge labels in
Reid's recipe. %This formula characterizes equality at the lower endpoint; see Remark~\ref{rmk:refined-count}.

The counting argument also has geometric consequences.
If $\mathbb C^3/G$ has an isolated singularity, then
$\mathcal B_0\geq\frac12$, and this bound is sharp (see Theorem~\ref{thm:geometric-bounds}).
On the other hand, for the family of isolated cyclic quotients of type
\[
G_m=\frac1{m^2+m+1}(1,m,m^2),\qquad m\geq2,
\]
we prove that
\[
\mathcal B_0(G_m)
\geq1-\frac{3(m-1)}{m(m+1)}
\longrightarrow1
\]
(see Proposition~\ref{prop:isolated-upper-bound}).  Thus the supremum remains $1$ even among isolated cyclic quotients. For an arbitrary finite abelian $G$, the equality $\mathcal B_0=1$ holds precisely when the fiber of the Hilbert--Chow morphism over the origin has dimension at most one (see Theorem~\ref{thm:geometric-bounds}).

Writing $n_0=(r-1)\mathcal B_0$, we further obtain
\[
b_4(Y)\leq n_0,
\qquad
\operatorname{rank}\operatorname{Pic}(Y)\leq2n_0,
\]
with both equalities when $\mathcal B_0=\frac13$.
(see Proposition~\ref{prop:Picard-rank}).  In particular, this yields an explicit upper bound on the top non-zero Betti number $b_4(Y)$ of the resolution in terms of the proportion $\mathcal{B}_0$.

By the characterization of coordinate walls of the $G$-Hilbert chamber established in \cite{BCQ15, CIK18, CLY26}, with explicit descriptions detailed in \cite{W20}, the integer $n_0$ counts the number of coordinate facets. Consequently, this chamber possesses at least $\lceil(r-1)/3\rceil$ coordinate facets in general, a lower bound that sharpens to $(r-1)/2$ when the quotient singularity is isolated.  To the best of our knowledge, these sharp uniform lower bounds have not previously appeared in the literature.

We conclude by discussing the distribution of the proportion $\mathcal B_0$ in cyclic families and its extension to nonabelian groups.

The remainder of this paper is organized as follows. Section~\ref{sec:McKay} recalls the toric description of $G$-Hilbert schemes and reviews Reid's recipe. In Section~\ref{sec:main}, we establish the sharp universal bounds on $\mathcal{B}_0$ and present the refined counting formula. Section~\ref{sec:geometric-interpretations} is devoted to isolated quotient singularities: we construct an explicit family approaching the upper bound $\mathcal{B}_0 \to 1$ and relate $\mathcal{B}_0$ to the Picard rank, number of compact exceptional divisors, and GIT coordinate facets.  Finally, Section~\ref{sec:conclusion} outlines future research directions, investigating the extrema and range of $\mathcal{B}_0$, along with the multiplicity of weight pairs realizing each value, for fixed order cyclic actions of the form $\frac{1}{r}(1,a,b)$ satisfying $a+b=r-1$. It also formulates the natural extension of the share $\mathcal B_0$ to finite non-abelian subgroups of $\operatorname{SL}_3(\mathbb{C})$.

\medskip
\textbf{Acknowledgement.} I am grateful to Alastair Craw for bringing \cite{CLY26} to my attention, outlining a potential connection, and pointing out an error in an earlier version of Example~\ref{example}.

\section{McKay correspondence}
\label{sec:McKay}

We start with a quick chronological overview of the subject. In \cite{McKay} John McKay observed that  for
a finite subgroup $G\subset \mathrm{SL}_2(\mathbb{C})$ there is a bijection

$$\left\{
\begin{aligned}
&\mbox{nontrivial~irreducible}\\
&\mbox{representations~of  }G
\end{aligned}\right\}\overset{1:1}{\longleftrightarrow}
\left\{
\begin{aligned}
&	\mbox{irreducible ~components}\\
&	\mbox{of ~the ~central~ fiber } \rho^{-1}(0)
\end{aligned}\right\},$$

where $\rho:Y\rightarrow \mathbb{C}^2/G$ is the minimal resolution of singularities. Notice that $R(G)$, the ring of representations of $G$,  is naturally isomorphic to $K_G(\mathbb{C}^2)$, the Grothendieck group of
$G$-equivariant coherent sheaves on $\mathbb{C}^2$. Following this observation, in \cite{GSV}, the McKay correspondence was realized
geometrically  as an isomorphism of Grothendieck groups $K_G(\mathbb{C}^2)\cong K_0(\mathrm{Coh}(Y))$. Next in \cite{KV} this isomorphism was lifted to an equivalence of triangulated categories of coherent sheaves: 
\begin{equation}
    \label{eq:derMcKay_n=2}
    D^b_G(\mathrm{Coh}(\mathbb{C}^2))\simeq D^b(\mathrm{Coh}(Y)).
\end{equation}
In particular under this equivalence, $\chi^!:=\chi\otimes\mathcal{O}_0$,
the skyscraper sheaf at $0$ associated to a nontrivial irreducible $G$-representation $\chi$, is mapped to the structure sheaf
of the corresponding exceptional divisor (twisted by $\mathcal{O}(-1)$). 

In 2001, Bridgeland, King, and Reid \cite{BKR} established an equivalence of derived categories $\mathrm{D}^b_G(\mathrm{Coh}(\mathbb{C}^3)) \simeq \mathrm{D}^b(\mathrm{Coh}(Y))$ for any finite subgroup $G \subset \operatorname{SL}_3(\mathbb{C})$, where $Y = G\text{-}\operatorname{Hilb}(\mathbb{C}^3)$ denotes the $G$-Hilbert scheme. They showed that $Y$ provides a crepant resolution of the quotient singularity $\mathbb{C}^3/G$. For abelian $G \subset \mathrm{SL}_3(\mathbb{C})$, the images of the character sheaves $\chi^!$ under the Bridgeland--King--Reid equivalence were established to be concentrated in a single cohomology degree in the case of isolated singularities by Craw and Ishii \cite{CL}, and subsequently for arbitrary finite abelian subgroups in \cite{CCL}. We briefly recall the background and notation.
\subsection{\texorpdfstring{$G\operatorname{-}\operatorname{Hilb}(\mathbb{C}^3)$}{G-Hilb(C3)} as a toric variety} 

For a detailed introduction to the construction of the $G$-Hilbert scheme as a toric variety, we refer the reader to Section 2 of \cite{C1}; see also Section 2 of \cite{Tsv}.

Let $G\subset \mathrm{SL}_3(\mathbb{C})$ be a finite abelian subgroup of order $r=|G|$, and let $\zeta=e^{2\pi i/r}$ be a primitive root of unity. After diagonalizing the action of $G$, we denote the corresponding coordinates on $\mathbb{C}^3$ by $x,y,z$.  

The lattice of exponents of Laurent monomials in $x,y,z$ is 
\(
L=\mathbb{Z}^3,
\)
with dual lattice $L^{\vee}$. For each group element $g=\mathrm{diag}(\zeta^{\gamma_1},\zeta^{\gamma_2},\zeta^{\gamma_3})$, we associate the vector 
\[
v_g=\frac{1}{r}(\gamma_1,\gamma_2,\gamma_3),\, 0\leq\gamma_i<r.
\]
Define
\[
N := L^{\vee} + \sum_{g\in G}\mathbb{Z}\cdot v_g, \qquad N_{\mathbb{R}}=N\otimes_{\mathbb{Z}}\mathbb{R},
\]
and let $M := \mathrm{Hom}_{\mathbb{Z}}(N, \mathbb{Z}) \subset L$ be the dual lattice of $G$-invariant Laurent monomials. The categorical quotient
\[
X = \mathrm{Spec}\,\mathbb{C}[x,y,z]^G
\]
is then identified with the affine toric variety $\mathrm{Spec}\,\mathbb{C}[\sigma^{\vee} \cap M]$, where $\sigma = \mathbb{R}_{\geq 0}^3 \subset N_{\mathbb{R}}$ is the standard positive octant generated by the basis $\{e_x, e_y, e_z\}$.

\begin{defn}
The \textit{junior simplex} $\Delta \subset N_{\mathbb{R}}$ is the convex hull of the basis vectors $e_x = (1,0,0)$, $e_y = (0,1,0)$, and $e_z = (0,0,1)$. The lattice points in $\Delta$ are given by the intersection $N \cap \Delta$.
\end{defn}

A subdivision of the cone $\sigma$ determines a fan $\Sigma$, and hence a toric variety $X_\Sigma$, together with a birational morphism $X_\Sigma\to X$.  

If the fan $\Sigma$ corresponds to a regular triangulation of the junior simplex into basic triangles, then the induced toric morphism $\pi \colon X_\Sigma \to X$ is a crepant resolution. In this case, $\Sigma$ is uniquely determined by the triangulation of the junior simplex; by a slight abuse of notation, we shall identify the fan $\Sigma$ with this triangulation.

\begin{defn}
A \textit{$G$-cluster} is a $G$-invariant zero-dimensional subscheme $Z\subset \mathbb{C}^n$ such that $H^0(\mathcal{O}_Z)$ is isomorphic, as a $\mathbb{C}[G]$-module, to the regular representation of $G$. The \textit{$G\operatorname{-}$Hilbert scheme} is the variety $Y=G\operatorname{-}\mathrm{Hilb}(\mathbb{C}^n)$, which serves as the fine moduli space parameterizing $G$-clusters.
\end{defn}

For finite abelian $G\subset SL_3(\mathbb C)$, the
$G$-Hilbert scheme is a smooth toric variety, and its
fan is determined by a triangulation of the junior
simplex into basic triangles.  For any finite subgroup $G\subset \mathrm{SL}_2(\mathbb{C})$ or $\mathrm{SL}_3(\mathbb{C})$, the natural morphism $Y\to X$ is a crepant resolution of singularities (see \cite{KV,BKR}). The partition of the junior simplex into basic triangles, which determines the fan of $Y$, can be computed via the three-step procedure described in \cite{CR} (see also \cite{Tsv}).

\subsection{Known results for abelian subgroups \texorpdfstring{$G \subset \mathrm{SL}_3(\mathbb{C})$}{G subset of SL3(C)}}
\label{subsec:BKR_reminder}

Let $\mathcal{Z} \subset Y \times \mathbb{C}^3$ denote the universal family. Consider the natural projections

\[\begin{tikzcd}
	& {\mathcal{Z}} \\
	Y && {\mathbb{C}^3}
	\arrow["\pi_Y"', from=1-2, to=2-1]
	\arrow["\pi_{\mathbb C^3}", from=1-2, to=2-3]
\end{tikzcd}\]

The derived equivalence established by Bridgeland, King, and Reid \cite{BKR} is provided by the functor $\Phi\colon \mathrm{D}^b(\mathrm{Coh}(Y)) \to \mathrm{D}^b_G(\mathrm{Coh}(\mathbb{C}^3)) $ defined by
\[
\Phi(E)
\coloneqq
\mathbf R\pi_{\mathbb C^3*}
\left(
\pi_Y^*E \overset{\mathbf L}{\otimes} \mathcal O_{\mathcal Z}
\right).
\]
Its inverse $\Psi = \Phi^{-1}$ is given by
\begin{equation}
    \label{eq:BKRequiv}
    \Psi(F)
\coloneqq
\left[
\mathbf R\pi_{Y*}
\left(
\pi_{\mathbb C^3}^*F
\overset{\mathbf L}{\otimes}
\mathcal O_{\mathcal Z}^{\vee}[3]
\right)
\right]^G,
\end{equation}
where
\[
\mathcal O_{\mathcal Z}^{\vee}
:=
\mathbf R\mathcal Hom_{Y\times\mathbb C^3}
\left(
\mathcal O_{\mathcal Z},
\mathcal O_{Y\times\mathbb C^3}
\right)
\]
(see \cite[Section~2.4]{CI} and
\cite[Section~2.4]{CCL}).

The following result appeared in the celebrated paper of Bridgeland, King and Reid (see \cite{BKR}).

\begin{thm}
	Let $G\subset \mathrm{SL}_3({\mathbb{C}})$ be a finite subgroup.
	\begin{enumerate}
		\item The variety $Y$ is a smooth irreducible variety, and the natural morphism $\pi \colon Y \to X$ is a crepant resolution of singularities.
    \item The Fourier--Mukai transform $\Psi \colon \mathcal{D}^b(Y) \to \mathcal{D}_G^b(\mathbb{C}^3)$ defined in \eqref{eq:BKRequiv} is an equivalence of bounded derived categories.
	\end{enumerate}
\label{BKRmain}
\end{thm}
Let $\chi$ be an irreducible representation  of $G$. There are two natural $G$-equivariant sheaves on $\mathbb{C}^3$ associated to $\chi$:
\begin{itemize}
	\item $\widetilde{\chi}:=\chi\otimes \mathcal{O}_{\mathbb{C}^3}$
	\item $\chi^!:=\chi\otimes\mathcal{O}_0$,
\end{itemize}
 where $\mathcal{O}_0=\mathcal{O}_{\mathbb{C}^3}/\mathfrak{m}_0= \mathbb{C}[x,y,z]/(x,y,z)$ is the structure sheaf of the origin in $\mathbb{C}^3$.

The image $\Psi(\widetilde{\chi})$  is isomorphic to  $L_\chi^{\vee}$, where $L_\chi$ is the corresponding tautological vector bundle of rank $\dim(\chi)$. This bundle is usually called a tautological or GSp-V sheaf (after Gonzales-Sprinberg and Verdier, see \cite{GSV}). The tautological vector bundles on $Y$ are defined as direct summands in the decomposition  
\begin{equation}
    \label{eq:taut_bundles_defn}
   \pi_{Y*}(\mathcal{O}_{\mathcal{Z}})=\bigoplus\mathcal{L}_\chi\otimes\chi,
\end{equation}
 with respect to the trivial $G$-action on $Y$.

However, the images of skyscraper sheaves $\chi^!$ under the equivalence are more intricate to describe. In the case where $G$ is abelian, the main result of \cite{CCL} provides an explicit characterization.

\begin{thm}
\label{CCLthm}
Let $G \subset \mathrm{SL}_3(\mathbb{C})$ be a finite abelian subgroup, and let $\chi \in G^\vee \setminus \{1\}$ be a nontrivial character. Exactly one of the following holds:
\begin{enumerate}
  \item The character $\chi$ is a vertex label at an interior vertex $v$, and 
  \[
  \Psi(\chi^!) \simeq \mathcal{L}_\chi^{-1} \otimes \mathcal{O}_{E_v}.
  \]
  \item The character $\chi$ labels exactly one interior edge $\tau$, and 
  \[
  \Psi(\chi^!) \simeq \mathcal{L}_\chi^{-1} \otimes \mathcal{O}_{C_\tau}.
  \]
  \item The character $\chi$ labels at least two interior edges. In this case, $\Psi(\chi^!)$ has exactly one nonzero cohomology sheaf, concentrated in degree $-1$.
\end{enumerate}
Here, $E_v$ and $C_\tau$ denote the exceptional divisor and curve corresponding to $v$ and $\tau$, respectively.
\end{thm}

\begin{rmk}
    Reid's recipe assigns one nontrivial character to each
interior edge and one or two nontrivial characters to
each interior vertex; see \cite[Section~2.3]{CCL}.
A character used as a vertex label labels exactly one
vertex and does not occur as an edge label
\cite[Corollary~4.6]{C1}.
\end{rmk}
% {\renewcommand{\arraystretch}{1.8}	\begin{table}[ht]
%		\begin{center}
%			\resizebox{16cm}{!}{
%				\begin{tabular}{ |c|c|c|c|} 
%					\hline
%					Reid's recipe& $H^{-2}(\Psi(\chi^!))$& $H^{-1}(\Psi(\chi^!))$& $H^{0}(\Psi(\chi^!))$ \\ 
%					\hline
%					$\chi$ marks a single divisor $E$& $0$& $0$& $\mathcal{L}^{-1}_{\chi}\otimes\mathcal{O}_E$ \\ 
%					\hline
%					$\chi$ marks a single curve $C$& $0$& $0$& $\mathcal{L}^{-1}_{\chi}\otimes\mathcal{O}_C$ \\ 
%					\hline
%					\pbox{13cm}{$\chi$ marks a chain of divisors\\ starting at $E$ and terminating at $F$}& $0$& $\mathcal{L}^{-1}_{\chi}(-E-F)\otimes\mathcal{O}_Z$& $0$ \\ 
%					\hline
%					\pbox{13cm}{$\chi$ marks three chains of divisors, starting at \\ $E_x$, $E_y$ and $E_z$ and meeting at a divisor $P$}& $0$& $\mathcal{L}^{-1}_{\chi}(-E_x-E_y-E_z)\otimes\mathcal{O}_{\mathcal{V}_Z}$& $0$ \\ 
%					\hline
%				%	$\chi=\chi_0$ & $w_{ZF_2}$& $w_{ZF_1}(ZF_2)$& $0$ \\ 	\hline
%			\end{tabular}}
%			\caption{Images of $\chi\otimes\mathcal{O}_0$}
%		\end{center}
%\end{table}}

\begin{rmk}
		A priori, each object $\Psi(\chi^!)$ is an abstract complex in $D^b(\mathrm{Coh}(Y))$. It follows from Theorem \ref{CCLthm}  that for every nontrivial character $\chi$ the object $\Psi(\chi^!)$ is a pure sheaf (a coherent sheaf or a coherent sheaf shifted by \([1]\)).
\end{rmk}

\section{Main results}
\label{sec:main}

We now turn to the question of bounding the parameter $\mathcal{B}_0$ (see Definition \ref{ShareDefn}).  
The following result provides sharp universal bounds.

\begin{thm}\label{MainThm}
Let $G\subset \mathrm{SL}_3(\mathbb{C})$ be a nontrivial finite
abelian subgroup. Then
\[
\frac{1}{3}\leq \mathcal{B}_0\leq 1,
\]
and both bounds are attained.
\end{thm}

\begin{proof}
Set $r=|G|$, and let $\Sigma$ denote the triangulation of
the junior simplex $\triangle$ corresponding to
$Y=G\operatorname{-}\mathrm{Hilb}(\mathbb{C}^3)$.
Since $Y$ is smooth, every triangle of $\Sigma$ is basic.
Moreover,
\[
[N:L^\vee]=r,
\qquad
\mathrm{Vol}_N\bigl(\operatorname{conv}(0,\triangle)\bigr)=r.
\]
Consequently, $\Sigma$ consists of exactly $r$ triangles.

Let $I$, $b$, and $e$ denote the numbers of interior vertices, boundary edges, and interior edges of the triangulation $\Sigma$, respectively. An edge of $\Sigma$ is a \textit{boundary edge} if it lies entirely within the boundary $\partial\triangle$ of the junior simplex; otherwise, it is an \textit{interior edge}. Since the boundary $\partial\triangle$ forms a subdivided triangle, it consists of $b$ vertices and $b$ edges. Letting $V$ and $E$ denote the total numbers of vertices and edges of $\Sigma$, respectively, we have $V = I + b$ and $E = e + b$. As the junior simplex is homeomorphic to a closed disc, and thus has Euler characteristic $1$, applying Euler's formula with $F = r$ faces yields
\begin{equation}
    \label{eq:Euler}
    (I + b) - (e + b) + r = 1
    \quad\implies\quad
    e = r + I - 1.
\end{equation}

Write
\begin{equation}
    \label{eq:n_0B_0}
 n_0\coloneqq(r-1)\mathcal{B}_0,
\end{equation}

and let $n_{-1}$ denote the number of nontrivial characters
$\chi$ for which
$H^{-1}(\Psi(\chi^!))\neq 0$.
By Theorem~\ref{CCLthm},
\begin{equation}
    \label{eq:n_0andn_{-1}}
n_0+n_{-1}=r-1.
\end{equation}

Reid's recipe marks every interior vertex with at least one
nontrivial character, and no character marks two distinct
interior vertices.  In addition, by \cite[Corollary~4.6]{C1}, a character used as a vertex label labels exactly one vertex.
The $I$ interior vertices therefore provide at least $I$
distinct characters marking individual divisors.
The single divisor case of Theorem~\ref{CCLthm} shows that
all these characters contribute to $n_0$, implying that 
\begin{equation}
    \label{eq:vertexInequality}
    I\leq n_0.
\end{equation}

On the other hand, every character contributing to $n_{-1}$
marks at least two interior edges. Since each interior edge
has exactly one character label, we obtain
\begin{equation}
    \label{eq:edgeInequality}
    2n_{-1}\leq e.
\end{equation}
Combining \eqref{eq:Euler}, \eqref{eq:vertexInequality} and \eqref{eq:edgeInequality}   yields
\[
2(r-1-n_0)
\leq r+I-1
\leq r+n_0-1.
\]
Therefore $3n_0\geq r-1$, and hence \eqref{eq:n_0B_0} gives the lower bound
\[
\mathcal{B}_0=\frac{n_0}{r-1}\geq\frac13.
\]
The upper bound follows immediately from the definition
of $\mathcal{B}_0$.

To establish sharpness of the lower bound, consider the group $G=\mathbb{Z}/4\mathbb{Z}$ with cyclic action of type 
\(
v_1=\frac14(1,1,2).
\)
%generated by $\operatorname{diag}(i,i,-1)$.
The junior simplex contains the interior lattice point $v_1$
and the additional boundary lattice point
\(
w=\frac12(1,1,0).
\)
The four interior edges of the $G$-Hilb triangulation are
\(
e_zv_1,\,v_1w,\,e_xv_1,\,e_yv_1.
\)

Index the characters by $\chi_j=\chi_1^j$, choosing
$\chi_1$ to be the character of $x$ and $y$. Then the first two edges are marked by $\chi_1$, while the last two are marked by $\chi_2$. The unique interior vertex is marked by $\chi_3$ (see Figure~\ref{JuniorSimplex2x2}). Theorem~\ref{CCLthm} therefore gives
\[
n_0=1\text{ and }
\mathcal{B}_0=\frac13.
\]

For sharpness of the upper bound, let
\begin{equation}
    \label{eq:varphi2,2}
    \varphi:
\mathbb{Z}/2\mathbb{Z}\times\mathbb{Z}/2\mathbb{Z}
\hookrightarrow \mathrm{SL}_3(\mathbb{C})
\end{equation}

be defined by
\[
\varphi(1,0)=
\begin{pmatrix}
-1&0&0\\
0&1&0\\
0&0&-1
\end{pmatrix},
\qquad
\varphi(0,1)=
\begin{pmatrix}
1&0&0\\
0&-1&0\\
0&0&-1
\end{pmatrix}.
\]
Each of the three nontrivial characters marks a single
interior edge of the corresponding triangulation
(see Figure~\ref{JuniorSimplex2x2}).
By Theorem~\ref{CCLthm}, all three characters contribute
to $n_0$, so
\[
\mathcal{B}_0=\frac33=1.
\]
\end{proof}

\begin{figure}[htbp!]
	\begin{center}
		
		\begin{tikzpicture}[scale=0.85]
			\draw (0,0) -- (6,0) -- (3,5.2) -- (0,0) -- cycle;
			
			\draw (0,0) -- (9/4,5.2/4);
			\draw (3,5.2) -- (9/4,5.2/4);
			\draw (6,0) -- (9/4,5.2/4);
			\draw (9/4,5.2/4) -- (9/2,5.2/2);
			
			\node at (9/4,5.2/4) {$\bullet$};
			\node at (9/2,5.2/2) {$\bullet$};
            \node at (4.7,2.8) {$w$};
			
			\node at (1.0,0.85) {$\mathsmaller{\chi_1}$};
			\node at (3.45,2.2) {$\mathsmaller{\chi_1}$};
			\node at (2.35,3.25) {$\mathsmaller{\chi_2}$};
			\node at (4.15,0.9) {$\mathsmaller{\chi_2}$};
			\node at (2.25,0.95) {$\mathsmaller{\chi_3}$};
            \node at (2.05,1.5) {$v_1$};
			
			\node at (-0.2,-0.2) {$e_z$};
			\node at (6.2,-0.2) {$e_y$};
			\node at (3,5.4) {$e_x$};
		\end{tikzpicture}
		\begin{tikzpicture}[scale=0.85]
			\draw (0,0) -- (6,0) -- (3,5.2) -- (0,0) -- cycle;
			\draw (4.5,2.6) -- (1.5,2.6);
			\draw (3,0) -- (1.5,2.6);
			\draw (4.5,2.6) -- (3,0);
		
			\node at (3,0) {$\bullet$};
			\node at (4.5,2.6) {$\bullet$};	
			\node at (1.5,2.6) {$\bullet$};	
			\node at (2,1) {$\mathsmaller{\chi_{11}}$};

			\node at (3,2.8) {$\mathsmaller{\chi_{10}}$};
			\node at (4.1,1) {$\mathsmaller{\chi_{01}}$};
			\node at (-0.2,-0.2) {$e_z$};
			\node at (6.2,-0.2) {$e_y$};
			\node at (3,5.4) {$e_x$};	
		\end{tikzpicture}
		
		\caption{The fan $\Sigma$ and character labels for the group $\mathbb{Z}/4\mathbb{Z}$ corresponding to $v_1=\dfrac{1}{4}(1,1,2)$ and the group $\mathbb{Z}/2\mathbb{Z}\times\mathbb{Z}/2\mathbb{Z}$ acting via the embedding $\varphi$ given in \eqref{eq:varphi2,2}.}
		\label{JuniorSimplex2x2}
	\end{center}
\end{figure}

We conclude this section by presenting two illustrative families of resolutions together with a few straightforward but useful observations.

\begin{ex}
\label{example}
Let $k\geq2$, and consider the cyclic groups
$G=\mathbb{Z}/2k\mathbb{Z}$ and
$\widetilde{G}=\mathbb{Z}/(2k+1)\mathbb{Z}$,
whose actions are specified by the respective generator
weight vectors
\[
\nu_1=\frac{1}{2k}(1,1,2k-2),
\qquad
\widetilde{\nu}_1=\frac{1}{2k+1}(1,1,2k-1).
\]
The corresponding lattices are
\[
N=L^{\vee}+\mathbb{Z}\cdot\nu_1,
\qquad
\widetilde{N}=L^{\vee}+\mathbb{Z}\cdot\widetilde{\nu}_1.
\]
In each case, index the characters so that $x$ and $y$
have character $\chi_1$, and write $\chi_j=\chi_1^j$.
The labels of the interior edges are distributed as follows:
\begin{itemize}
  \item if $r=2k$, then $k$ edges are marked with $\chi_1$,
  and for each $1\leq i\leq k-1$, two edges are marked
  with $\chi_{2i}$;
  \item if $r=2k+1$, then $k+2$ edges are marked with
  $\chi_1$, and for each $1\leq i\leq k-1$, two edges
  are marked with $\chi_{2i+1}$
\end{itemize}
(see Figure~\ref{JuniorSimplex67}). 
By Theorem~\ref{CCLthm}, exactly these characters
have transforms concentrated in degree $-1$.
Since $G$ and $\widetilde G$ have $2k-1$ and $2k$
nontrivial characters, respectively, we obtain
\[
\mathcal B_0(G)
=1-\frac{k}{2k-1}
 =\frac{k-1}{2k-1},\quad
\mathcal B_0(\widetilde G)
=1-\frac{k}{2k}
 =\frac12.
\]
\end{ex}

		\begin{figure}[htbp!]
			\begin{center}
				
				\begin{tikzpicture}[scale=1]
					\draw (0,0) -- (6,0) -- (3,5.2) -- (0,0) -- cycle;
					\draw (0,0) -- (4.5,2.6);
					\draw (6,0) -- (1.5,5.2/6);
					\draw (6,0) -- (3,5.2/3);
					\draw (3,5.2) -- (1.5,5.2/6);
					\draw (3,5.2) -- (3,5.2/3);
					\node at (4.5,2.6) {$\bullet$};	
					\node at (1.5,5.2/6) {$\bullet$};
					\node at (1,5.2/6) {$\mathsmaller{\chi_1}$};
					\node at (2.3,1.6) {$\mathsmaller{\chi_1}$};
					\node at (3.7,2.4) {$\mathsmaller{\chi_1}$};
					\node at (3.25,3.5) {$\mathsmaller{\chi_2}$};
					\node at (4.15,1.3) {$\mathsmaller{\chi_2}$};
					\node at (2.65,3.5) {$\mathsmaller{\chi_4}$};
					\node at (3.2,0.7) {$\mathsmaller{\chi_4}$};
					\node at (3,5.2/3) {$\bullet$};	
					\node at (-0.2,-0.2) {$e_z$};
					\node at (6.2,-0.2) {$e_y$};
					\node at (3,5.4) {$e_x$};	
				\end{tikzpicture}
				\begin{tikzpicture}[scale=1]
					\draw (0,0) -- (6,0) -- (3,5.2) -- (0,0) -- cycle;
					\draw (0,0) -- (27/7,15.6/7);
					\draw (6,0) -- (9/7,5.2/7);
					\draw (6,0) -- (18/7,10.4/7);
					\draw (6,0) -- (27/7,15.6/7);
					\draw (3,5.2) -- (9/7,5.2/7);
					\draw (3,5.2) -- (18/7,10.4/7);
					\draw (3,5.2) -- (27/7,15.6/7);
					\node at (9/7,5.2/7) {$\bullet$};
					\node at (18/7,10.4/7) {$\bullet$};	
					\node at (27/7,15.6/7) {$\bullet$};	
					\node at (0.8,5.2/6-0.17) {$\mathsmaller{\chi_1}$};
					\node at (2,1.4) {$\mathsmaller{\chi_1}$};
					\node at (3.3,2.11) {$\mathsmaller{\chi_1}$};
					\node at (3.3,3.5) {$\mathsmaller{\chi_1}$};
					\node at (4.37,1.3) {$\mathsmaller{\chi_1}$};
					\node at (2.6,3.1) {$\mathsmaller{\chi_3}$};
					\node at (3.7,0.777) {$\mathsmaller{\chi_3}$};
					\node at (1.85,2.7) {$\mathsmaller{\chi_5}$};
					\node at (2.7,0.3) {$\mathsmaller{\chi_5}$};
					\node at (-0.2,-0.2) {$e_z$};
					\node at (6.2,-0.2) {$e_y$};
					\node at (3,5.4) {$e_x$};	
				\end{tikzpicture}
				
				\caption{$\Sigma$ fans and character labels for $G=\dfrac{1}{6}(1,1,4)$ and $\widetilde{G}=\dfrac{1}{7}(1,1,5)$}
				\label{JuniorSimplex67}
			\end{center}
		\end{figure}

\begin{rmk}\label{rmk:refined-count}
Retain the notation $I,e,n_0,n_{-1}$ from the proof of
Theorem~\ref{MainThm}, and write
\[
\mathfrak{H}_{-1}
\coloneqq
\left\{
\chi\in G^\vee\setminus\{1\}
\;\middle|\;
H^{-1}\bigl(\Psi(\chi^!)\bigr)\neq0
\right\}.
\]
Let $q$ denote the number of interior vertices carrying
two vertex labels in Reid's recipe, and let $c$ denote
the number of characters labelling exactly one interior
edge. For $\chi\in\mathfrak{H}_{-1}$, let $m_\chi$ be the
number of interior edges labelled by $\chi$, and set
\[
D\coloneqq\sum_{\chi\in\mathfrak{H}_{-1}}(m_\chi-2).
\]
By Theorem~\ref{CCLthm}, every $m_\chi\geq2$, so $D\geq0$.

The interior  vertices carrying two labels correspond precisely to
the compact del Pezzo surfaces of degree six in the exceptional fiber
\cite[Section~2.3]{CCL}.
Moreover, by \cite[Corollary~4.6]{C1}, a character used
as a vertex label labels exactly one vertex and does not
also occur as an edge label.
Theorem~\ref{CCLthm} therefore gives
\begin{equation}
\label{eq:refined-counts}
n_0=I+q+c,
\qquad
e=c+\sum_{\chi\in\mathfrak{H}_{-1}}m_\chi
  =c+2n_{-1}+D.
\end{equation}
Combining these identities with \eqref{eq:Euler} and \eqref{eq:n_0andn_{-1}}
yields
\[
3n_0=(r-1)+q+2c+D.
\]
Consequently,
\begin{equation}
\label{eq:refined-B0}
\mathcal{B}_0
=
\frac13+\frac{q+2c+D}{3(r-1)}.
\end{equation}
In particular,
\[
\mathcal{B}_0=\frac13
\quad\Longleftrightarrow\quad
q=c=D=0.
\]
Thus equality holds precisely when every interior vertex
carries exactly one vertex label, no character labels
exactly one interior edge, and every character in
$\mathfrak{H}_{-1}$ labels exactly two interior edges.
\end{rmk}

%\begin{rmk}
%\begin{enumerate}
%    \item The number of internal edges equals $\tfrac{3r-3}{2}$ if and only if each side of the junior simplex is itself an edge of the triangulation. This condition is equivalent to requiring $\gcd(a,r)=\gcd(b,r)=1$.

%    \item By Theorem~\ref{CCLthm}, $H^{0}(\Psi(\chi^!))$ is nonzero whenever $\chi$ marks at most one internal edge. Consequently, the number of internal edges marked with characters $\chi$ such that $H^{0}(\Psi(\chi^!))\neq 0$ is bounded above by $r-1$. In particular, if the number of internal edges exceeds $r-1$, it follows that $\mathcal{B}_0$ is strictly less than $1$. 
%\end{enumerate} 
%\end{rmk}

%\begin{figure}[htbp!]
%	\begin{center}
		
%		\begin{tikzpicture}[scale=1]
%			\draw (0,0) -- (6,0) -- (3,5.2) -- (0,0) -- cycle;
			
%			\node at (12/3,10.4/3) {$\bullet$};
%			\node at (15/3,5.2/3) {$\bullet$};
			
%			\node at (1.5,0) {$\bullet$};
%			\node at (3,0) {$\bullet$};
%			\node at (4.5,0) {$\bullet$};
			
%			\node at (24/12,20.8/12) {$\bullet$};
%			\node at (42/12,20.8/12) {$\bullet$};
%			\node at (30/12,41.6/12) {$\bullet$};
		
%			\node at (-0.2,-0.2) {$e_z$};
%			\node at (6.2,-0.2) {$e_y$};
%			\node at (3,5.4) {$e_x$};	
%		\end{tikzpicture}

%		\caption{$\Sigma$ fan and character labels for  $G=\dfrac{1}{5}(1,1,3)$}
%		\label{JuniorSimplex5}
%	\end{center}
%\end{figure}

\section{Geometric Interpretations and Extensions}
\label{sec:geometric-interpretations}

We write
\begin{equation}
\label{eq:Hilb_Chow}
\rho:Y=G\operatorname{-}\mathrm{Hilb}(\mathbb{C}^3)
\longrightarrow \mathbb{C}^3/G
\end{equation}
for the Hilbert--Chow morphism and retain the notation $I,q,c$ from Remark~\ref{rmk:refined-count}.

\subsection{Isolated quotient singularities and the upper endpoint}

We now investigate how the geometry of the quotient singularity $\mathbb{C}^3/G$ constrains the proportion $\mathcal{B}_0$, establishing sharp lower bounds in the case of isolated singularities and providing a complete characterization of the upper endpoint $\mathcal{B}_0 = 1$.

\begin{thm}
\label{thm:geometric-bounds}
If $\mathbb{C}^3/G$ has an isolated singularity at the origin, then
\begin{equation}
    \label{eq:B_0for_isolated_sing}
    \mathcal{B}_0
=
\frac12+\frac{q+c}{r-1}
\geq\frac12.
\end{equation}
The lower bound is attained.

Without assuming that the singularity is isolated,
one has
\[
\mathcal{B}_0=1
\quad\Longleftrightarrow\quad
I=0
\quad\Longleftrightarrow\quad
\dim(\rho^{-1}(0))\leq1.
\]
\end{thm}

\begin{proof}
Suppose that $\mathbb{C}^3/G$ has an isolated singularity.
If the boundary of junior simplex $\triangle$ contained a lattice point $v$ other than its three corners, then, after permuting
coordinates, we could write
\[
v=\frac{1}{r}(a,r-a,0),
\qquad 0<a<r.
\]
This point corresponds to the nonidentity element
\[
g=\operatorname{diag}
\bigl(e^{2\pi ia/r},e^{-2\pi ia/r},1\bigr)\in G.
\]
In particular, $g$ has eigenvalue $1$, contradicting
the criterion for isolated quotient singularities
\cite[Theorem~A]{YY93}.
Thus the only boundary lattice points of $\triangle$
are its three corners, and the number $b$ of boundary edges is three.
Therefore
\[
3r=2e+b=2e+3.
\]

Combining this identity with \eqref{eq:Euler} gives 
\begin{equation}
    \label{eq:inside_vertices}
    r-1=2I.
\end{equation}

Using \eqref{eq:refined-counts}, we obtain
\[
\mathcal{B}_0
=
\frac{I+q+c}{r-1}
=
\frac12+\frac{q+c}{2I}.
\]
In particular, equality is attained by the isolated cyclic quotient
actions
\[
G=\frac{1}{2k+1}(1,1,2k-1),
\qquad k\geq2,
\]
from Example~\ref{example}.

We now consider the upper endpoint.
If $\mathcal{B}_0=1$, then Theorem~\ref{CCLthm} implies
that every character labels at most one interior edge.
Consequently, $e\leq r-1$.
Since $e=r+I-1$, this forces $I=0$.

Conversely, suppose that $I=0$.
Every nontrivial character then labels at least one
interior edge, since there are no interior vertices.
There are $r-1$ such characters and exactly
$e=r-1$ interior edges. Since every interior edge has
one character label, each character labels exactly one
interior edge. Theorem~\ref{CCLthm} therefore gives
$\mathcal{B}_0=1$.

Finally, the two-dimensional irreducible components of
$\rho^{-1}(0)$ are precisely the compact exceptional
divisors corresponding to interior vertices of the junior
simplex \cite[Section~2.2]{CCL}.
Thus $I=0$ if and only if $\dim(\rho^{-1}(0))\leq1$.
\end{proof}

\subsection{Isolated cyclic quotients approaching the upper bound}
\label{subsec:isolated-upper-bound}
Next, we construct a family of isolated cyclic quotient
singularities for which $\mathcal B_0$ tends to $1$.
This shows that the upper bound remains optimal even
within this class.
For an integer $m\geq2$, set
\[
r=m^2+m+1,
\qquad
G_m=\frac1r(1,m,m^2)\subset \mathrm{SL}_3(\mathbb C),
\qquad
Y_m=G_m\operatorname{-Hilb}(\mathbb C^3).
\]
Thus the coordinate functions $x,y,z$ have characters
$\chi_1,\chi_m,\chi_{m^2}$, respectively.

\begin{prop}
\label{prop:isolated-upper-bound}
The quotient $\mathbb C^3/G_m$ has an isolated singularity,
and
\[
1-\frac{3(m-1)}{m(m+1)}
\leq \mathcal B_0(G_m)\leq1.
\]
In particular,
\begin{equation}
    \label{eq:attaining1_inthelimit}
\lim_{m\to\infty}\mathcal B_0(G_m)=1
\end{equation}
and
\begin{equation}
    \label{eq:attainsup_inthelimit}
    \sup\left\{
\mathcal B_0(G)
\;\middle|\;
\begin{array}{c}
G\subset \mathrm{SL}_3(\mathbb C)
\text{ is nontrivial, finite and cyclic},\\
\mathbb C^3/G\text{ has an isolated singularity}
\end{array}
\right\}=1.
\end{equation}
\end{prop}

\begin{proof}
Since $\gcd(m,r_m)=1$, all three coordinate characters
are faithful. Thus $G_m$ acts freely on
$\mathbb C^3\setminus\{0\}$, and the quotient has an
isolated singularity. It follows from \eqref{eq:inside_vertices} that
\begin{equation}
    \label{eq:num_of_vertices}
    I=\frac{r_m-1}{2}=\frac{m(m+1)}2.
\end{equation}

Let $N_m$ denote the cocharacter lattice associated with the cyclic quotient $\mathbb{C}^3/G_m$. 
The ratios
\[
x^m : y, \qquad y^m : z, \qquad z^m : x
\]
are $G_m$-invariant, due to the relation $m^3 - 1 = (m-1)r_m$. 
In the Craw--Reid construction, these rational functions determine a \textit{regular triangle}
\(
T = \operatorname{conv}(P, Q, R),
\)
whose vertices are given by
\[
P = \frac{(1, m, m^2)}{r_m}, \quad
Q = \frac{(m^2, 1, m)}{r_m}, \quad
R = \frac{(m, m^2, 1)}{r_m}.
\]

The lattice $N_m$ is given by $N_m = \mathbb{Z}^3 + \mathbb{Z}P$. Setting
\[
u = \frac{Q - P}{m-1} = \frac{(m+1, -1, -m)}{r_m},
\qquad
v = \frac{R - P}{m-1} = \frac{(1, m, -m-1)}{r_m},
\]
we find that $u = (m+1)P - e_y - me_z$ and $v = P - e_z$, which ensures $u, v \in N_m$. 
Furthermore, the determinant $\det(P, u, v) = 1/r$ equals the covolume of $N_m$. 
Thus, $\{P, u, v\}$ forms a $\mathbb{Z}$-basis for $N_m$, and $\{u, v\}$ forms a basis for its height-zero sublattice. 
Therefore, the triangle $T$ is given by
\[
T = P + \operatorname{conv}\bigl(0, (m-1)u, (m-1)v\bigr).
\]

The lattice points strictly inside $T$ are
\[
P+iu+jv,
\qquad i,j\in\mathbb Z_{>0},\quad i+j<m-1.
\]
Their number is $(m-2)(m-3)/2$. By \cite[Theorem~1.2]{CR}, the $G_m$-Hilbert scheme triangulation restricts to the regular tessellation of $T$. At each lattice point strictly inside $T$, the six primitive edge directions are therefore $\pm u,\pm v,\pm(u-v)$. The corresponding compact toric divisor is therefore a del Pezzo surface of degree six
\cite[Section~2.8.4]{CR}, and it carries two vertex
labels in Reid's recipe \cite[Lemma~3.4]{C1}.
Hence
\[
q\geq\frac{(m-2)(m-3)}2.
\]
Combining $c \geq 0$ with equations \eqref{eq:num_of_vertices} and \eqref{eq:refined-counts}, we obtain
\[
n_0=I+q+c
\geq
\frac{m(m+1)}2+\frac{(m-2)(m-3)}2
=m^2-2m+3.
\]
Dividing by $r-1=m(m+1)$ yields
\[
\mathcal B_0(G_m)
\geq
1-\frac{3(m-1)}{m(m+1)}.
\]
Together with the universal upper bound $\mathcal{B}_0(G_m) \leq 1$, this establishes the assertions in \eqref{eq:attaining1_inthelimit} and \eqref{eq:attainsup_inthelimit}.
\end{proof}

\begin{figure}[htbp!]
    \centering
    \begin{minipage}[c]{0.48\linewidth}
        \centering
        \resizebox{\linewidth}{!}{%
        \begin{tikzpicture}[
            edge label/.style={fill=white,inner sep=0.6pt}
        ]
            \coordinate (ex) at (3,5.2);
            \coordinate (ey) at (6,0);
            \coordinate (ez) at (0,0);

            % A_{s,t}=P+su+tv in the original junior simplex.
            \foreach \s/\t in {0/0,1/0,2/0,3/0,0/1,1/1,2/1,0/2,1/2,0/3}{
                \coordinate (A\s\t) at
                ({(27+9*\s+27*\t)/21},
                 {5.2*(1+5*\s+\t)/21});
            }

            \fill[gray!10] (A00)--(A30)--(A03)--cycle;
            \draw (ez)--(ey)--(ex)--cycle;

            % Subdivision of T into nine basic triangles.
            \draw (A00)--(A30)--(A03)--cycle;

            % Exterior fans.
            \foreach \a in {A30,A21,A12,A03}
                \draw (ex)--(\a);
            \foreach \a in {A00,A01,A02,A03}
                \draw (ey)--(\a);
            \foreach \a in {A00,A10,A20,A30}
                \draw (ez)--(\a);

            % Exterior edge labels.
            \foreach \a/\b/\c in {
                ex/A30/16,ex/A21/12,ex/A12/8,ex/A03/4,
                ey/A00/16,ey/A01/11,ey/A02/6,ey/A03/1,
                ez/A00/4,ez/A10/3,ez/A20/2,ez/A30/1
            }{
                \path (\a)--(\b);
            }

            \foreach \a in {A00,A10,A20,A30,A01,A21,A02,A12,A03}
                \node at (\a) {$\bullet$};

            \node at (1.12,0.48) {$P$};
            \node at (2.83,4.04) {$Q$};
            \node at (5.35,1.12) {$R$};
            \node at (3,1.96) {$T$};

            \node at (-0.2,-0.2) {$e_z$};
            \node at (6.2,-0.2) {$e_y$};
            \node at (3,5.43) {$e_x$};
            \node at (3,-0.68)
                {\small (a) The junior simplex};
        \end{tikzpicture}%
        }
    \end{minipage}
    \hfill
    \begin{minipage}[c]{0.48\linewidth}
        \centering
        \resizebox{\linewidth}{!}{%
        \begin{tikzpicture}[
            edge label/.style={fill=white,inner sep=0.6pt},
            vertex label/.style={text=blue!65!black}
        ]
            % Enlarged copy of T, with P at the origin.
            \foreach \s/\t in {0/0,1/0,2/0,3/0,0/1,1/1,2/1,0/2,1/2,0/3}{
                \coordinate (A\s\t) at
                    ({\s+2*\t},{5.2*\s/3});
            }

            \fill[gray!10] (A00)--(A30)--(A03)--cycle;
            \draw (A00)--(A30)--(A03)--cycle;
            \draw (A10)--(A12) (A20)--(A21);
            \draw (A01)--(A21) (A02)--(A12);
            \draw (A10)--(A01) (A20)--(A02);

            % All eighteen edge labels on and inside T.
            \foreach \a/\b/\c in {
                A00/A10/16,A10/A20/16,A20/A30/16,
                A00/A01/4,A01/A02/4,A02/A03/4,
                A30/A21/1,A21/A12/1,A12/A03/1,
                A10/A11/3,A11/A12/3,A20/A21/2,
                A01/A11/12,A11/A21/12,A02/A12/8,
                A10/A01/11,A20/A11/6,A11/A02/6
            }{
                \path (\a)--(\b);
            }

            \foreach \a in {A00,A10,A20,A30,A01,A11,A21,A02,A12,A03}
                \node at (\a) {$\bullet$};

            % Vertex labels.
            \node at (-0.2,-0.34)
                {$P$};
            \node at (3,5.50)
                {$Q$};
            \node at (6.3,-0.34)
                {$R$};

            % The unique double-labelled vertex.
            \node at (2.6,1.52) {$S$};
            \node[vertex label] at (3.7,1.95)
                {$\mathsmaller{\chi_7},\mathsmaller{\chi_{14}}$};

            \node at (3,-0.87)
                {\small (b) The subdivision of $T$};
        \end{tikzpicture}%
        }
    \end{minipage}

    \caption{The triangulation and Reid's-recipe labels for
    $G_4=\frac1{21}(1,4,16)$, with an enlargement of $T$.
    The triangle $T$ consists of nine basic triangles, and
    its unique interior vertex
    $S=(1/3,1/3,1/3)$ carries the two labels
    $\chi_7,\chi_{14}$.}
    \label{fig:cyclic-upper-m4}
\end{figure}

\subsection{Coordinate walls of the Hilbert chamber}
Let $v=[\mathbb C[G]]\in R(G)$ be the class of the
regular representation, and write $\theta_\chi:=\theta([\chi])$ for $\chi\in G^\vee$.
Consider the $\mathbb{Q}$-vector space of stability parameters
\[
\Theta
=
\left\{
\theta \in \operatorname{Hom}_{\mathbb{Z}}\bigl(R(G), \mathbb{Q}\bigr)
\;\middle|\;
\theta(v) = 0
\right\},
\]
and its positive orthant
\begin{equation}
    \label{eq:positive_orthant}
    \Theta_+
\coloneqq
\left\{
\theta\in\Theta
\;\middle|\;
\theta_\chi>0
\text{ for every nontrivial irreducible representation }\chi \text{ of } G\right\}.
\end{equation}

Recall that a $G$-constellation is a $G$-equivariant
coherent sheaf on $\mathbb C^3$ whose global sections
form the regular representation. The GIT description
of the $G$-Hilbert scheme used below is reviewed in
\cite[Section~2.3]{CLY26}. Let $C\subseteq \Theta$ be the GIT chamber containing $\Theta_+$ such that the moduli space of stable $G$-constellations associated to any parameter in $C$ is isomorphic to the $G$-Hilbert scheme $Y$. Note that $C$ may strictly contain $\Theta_+$.

For each nontrivial character $\chi$, define the corresponding coordinate hyperplane
\[
H_\chi\coloneqq\{\theta\in\Theta\mid\theta_\chi=0\}.
\]
By \cite[Theorem~1.2 and Lemma~2.6]{CLY26}, we have
\[
H^0\bigl(\Psi(\chi^!)\bigr)\neq0
\quad\Longleftrightarrow\quad
H_\chi\cap C=\varnothing
\quad\Longleftrightarrow\quad
H_\chi\cap\overline{C}
\text{ is a facet of }\overline{C}.
\]
Thus the characters contributing to $\mathfrak{H}_0$ are exactly
those whose coordinate hyperplanes contain facets of
$\overline{C}$.

Defining the count of such walls by
\[
f_{\mathrm{coord}}(C)
\coloneqq
\#\left\{
\chi\in G^\vee\setminus\{1\}
\;\middle|\;
H_\chi\cap\overline{C}
\text{ is a facet of }\overline{C}
\right\},
\]
we obtain the equalities
\begin{equation}
\label{eq:chamber-B0}
f_{\mathrm{coord}}(C)=n_0,
\qquad
\mathcal{B}_0=\frac{f_{\mathrm{coord}}(C)}{r-1}.
\end{equation}

When the quotient singularity is isolated, Eq.~\eqref{eq:B_0for_isolated_sing} yields the sharper expression

%which quantifies the facets lying within the coordinate hyperplanes $H_\chi$.

\begin{equation}
    \label{eq:coord_hyp_lower_bound1}
f_{\mathrm{coord}}(C)
=
\frac{r-1}{2}+q+c
\geq\frac{r-1}{2}.
\end{equation}

%which quantifies the facets lying within the coordinate hyperplanes $H_\chi$.

\begin{prop}
\label{prop:coordinate-facets}
Let $C$ be the GIT chamber containing $\Theta_+$, and
let $f_{\mathrm{coord}}(C)$ be the number of facets of $\overline{C}$ contained in the coordinate hyperplanesv$H_\chi$, $\chi\neq1$. With the notation of Remark~\ref{rmk:refined-count},
one has
\[
f_{\mathrm{coord}}(C)
=
I+q+c
=
\frac{r-1+q+2c+D}{3}.
\]
In particular,
\begin{equation}
    \label{eq:coord_hyp_lower_bound2}
    f_{\mathrm{coord}}(C)
\geq
\left\lceil\frac{r-1}{3}\right\rceil,
\end{equation}

and
\[
3f_{\mathrm{coord}}(C)=r-1
\quad\Longleftrightarrow\quad
q=c=D=0.
\]
\end{prop}

\begin{proof}
By \eqref{eq:chamber-B0},
$f_{\mathrm{coord}}(C)=n_0$.
The formulas therefore follow from
\eqref{eq:refined-counts} and \eqref{eq:refined-B0}.
The bound and equality criterion follow from
$q,c,D\geq0$.
\end{proof}

\begin{rmk}
\label{rmk:wall-count-literature}
For finite abelian groups, the identification of degree-zero transforms with coordinate walls bounding the GIT chamber $C$ was established in \cite[Proposition~1.3]{BCQ15}; see also \cite[Theorem~4.17]{W20}. However, our counting argument provides a sharp uniform lower bound on the number of coordinate walls. To the best of our knowledge, the explicit lower bounds on $f_{\mathrm{coord}}(C)$ in \eqref{eq:coord_hyp_lower_bound1} and \eqref{eq:coord_hyp_lower_bound2} have not previously appeared in the literature.
\end{rmk}

\subsection{Picard rank and compact exceptional divisors}
\label{subsec:Picard-rank}

Let $G\subset \mathrm{SL}_3(\mathbb{C})$ be a nontrivial finite
abelian subgroup of order $r$, and let $\rho$ be the associated crepant resolution as in \eqref{eq:Hilb_Chow}. Write
\[
b_j(Y)\coloneqq\dim_{\mathbb{Q}}H^j(Y,\mathbb{Q})
\]
for the $j$-th Betti number of $Y$, set $p\coloneqq\operatorname{rk}\operatorname{Pic}(Y)$ for the Picard rank, and retain the notation of Remark~\ref{rmk:refined-count}.

By \cite[Section~2.2]{CCL}, the interior vertices of the triangulation of the junior simplex parametrize the compact exceptional prime divisors of $\rho$, whose fundamental classes form a basis of $H_4(Y,\mathbb{Z})$ \cite[Proposition~7.3]{C1}. Moreover, the description of
the Picard group and the cohomological McKay
correspondence yield
\begin{equation}
\label{eq:topological-counts}
b_4(Y)=I,
\qquad
p=b_2(Y)=r-1-I;
\end{equation}
see \cite[Propositions~6.3 and~7.1]{C1}.
Combining these topological identities with the refined
counting formula of Remark~\ref{rmk:refined-count}
gives the following geometric bounds.

\begin{prop}
\label{prop:Picard-rank}
One has
\begin{equation}\label{eq:geometric-counts}
n_0 = b_4(Y) + q + c,
\qquad
p = 2n_0 - c - D.
\end{equation}
Consequently, the Betti number $b_4(Y)$ and the Picard rank $p$ satisfy the bounds
\begin{equation}
\label{eq:geometric-bounds}
\max\{0, r - 1 - 2n_0\}
\leq b_4(Y) \leq n_0,
\qquad
p \leq 2n_0.
\end{equation}
Equality in the two upper bounds is characterized by
\[
b_4(Y) = n_0
\quad\Longleftrightarrow\quad q = c = 0,
\qquad
p = 2n_0
\quad\Longleftrightarrow\quad c = D = 0.
\]
In particular,
\begin{equation}
\label{eq:geometric-extremality}
\mathcal{B}_0 = \frac{1}{3}
\quad\Longleftrightarrow\quad
b_4(Y) = n_0
\ \text{and}\
p = 2n_0.
\end{equation}
\end{prop}

\begin{proof}
Combining the first identity in \eqref{eq:refined-counts} with the equality $I = b_4(Y)$, we obtain
\[
n_0 = I + q + c = b_4(Y) + q + c,
\]
which yields the first relation in \eqref{eq:geometric-counts}.  To establish the second identity in \eqref{eq:geometric-counts}, we substitute $e = (r-1) + I$ and $n_{-1} = r - 1 - n_0$ into
\[
e = c + 2n_{-1} + D.
\]
This substitution gives
\[
(r-1) + I = c + 2(r - 1 - n_0) + D,
\]
from which we deduce
\[
p = r - 1 - I = 2n_0 - c - D.
\]
As $q$, $c$, and $D$ are non-negative, these identities immediately imply the upper bounds asserted in \eqref{eq:geometric-bounds}. Furthermore, noting that
\[
b_4(Y) = r - 1 - p \geq r - 1 - 2n_0,
\]
we complete the proof of the remaining estimate in \eqref{eq:geometric-bounds}.

Finally, if $\mathcal{B}_0 = \frac{1}{3}$, Remark~\ref{rmk:refined-count} forces $q = c = D = 0$. Inserting these vanishing conditions into \eqref{eq:geometric-counts} yields $b_4(Y) = n_0$ and $p = 2n_0$. Conversely, assuming these two numerical equalities, we find
\[
r - 1 = p + b_4(Y) = 3n_0,
\]
assuring that $\mathcal{B}_0 = \frac{1}{3}$.
\end{proof}

Thus, at the extremal lower bound $\mathcal{B}_0 = \frac{1}{3}$, the $G$-Hilbert scheme $Y$ possesses precisely $n_0$ compact exceptional prime divisors and attains the maximal feasible Picard rank $2n_0$. More generally, the nonnegative quantities $q+c$ and $c+D$ measure the respective differences
\[
n_0-b_4(Y)=q+c,
\qquad
2n_0-\operatorname{rank}\operatorname{Pic}(Y)=c+D.
\]

\section{Future Directions and Open Problems}
\label{sec:conclusion}

We conclude by outlining several natural questions and directions for further research

\subsection{Distribution of \texorpdfstring{$\mathcal{B}_0$}{B0} values across cyclic subgroups}
Fix an integer $r\geq3$ and define the set of index pairs
\[
\mathcal S_r
=
\{(a,b)\in\mathbb Z_{>0}^2\mid a+b=r-1\}.
\]
For each $(a,b)\in\mathcal S_r$, let
\[
G_{a,b}
=
\left\langle
\operatorname{diag}(\zeta,\zeta^a,\zeta^b)
\right\rangle,
\qquad \zeta=e^{2\pi i/r},
\]
and consider the map
\[
\mathfrak G_r:\mathcal S_r
\longrightarrow
\left[\frac13,1\right]\cap\frac1{r-1}\mathbb Z,
\qquad
\mathfrak G_r(a,b)\coloneqq\mathcal B_0(G_{a,b}).
\]

\begin{rmk}
Note that  $\mathfrak{G}_r(a,b) = \mathfrak{G}_r(b,a)$.
\end{rmk}

This setup naturally leads to the following questions:
\begin{enumerate}
	\item What are the global maximum and minimum values of $\mathfrak{G}_r$? 
	\item What is the exact image (range) of $\mathfrak{G}_r$? 
	\item For each value $t\in \left[\frac13,1\right]\cap\frac1{r-1}\mathbb Z$, how many pairs
$(a,b)\in\mathcal S_r$ satisfy $\mathfrak G_r(a,b)=t$?
\end{enumerate}

\subsection{Extension to nonabelian groups}

Let now $G\subset \mathrm{SL}_3(\mathbb{C})$ be an arbitrary
nontrivial finite subgroup.
By \cite[Theorem~1.2]{CLY26}, for every nontrivial
irreducible representation $\varrho$ of $G$, the object
\[
\Psi(\mathcal{O}_0\otimes\varrho)
\]
has exactly one nonzero cohomology sheaf, occurring in
degree $0$ or $-1$.
We may therefore define
\[
\mathcal{B}_0(G)
=
\frac{
\#\left\{
\varrho\in\operatorname{Irr}(G)\setminus\{\mathbf{1}\}
\;\middle|\;
H^0\bigl(
\Psi(\mathcal{O}_0\otimes\varrho)
\bigr)\neq0
\right\}
}{
|\operatorname{Irr}(G)|-1
}.
\]
This agrees with the preceding Definition~\ref{ShareDefn} when $G$ is
abelian. %The denominator counts nontrivial irreducible representations rather than nonidentity group elements. The sign interpretation of the Gale dual matrix remains valid by \cite[Theorem~1.3]{CLY26}.

\begin{ex}
\label{ex:trihedral-B0}
Let $\zeta=e^{2\pi i/13}$, and consider the trihedral
subgroup \(G\subset SL_3(\mathbb C)\) of order $39$ generated by
\[
\begin{pmatrix}
\zeta&0&0\\
0&\zeta^3&0\\
0&0&\zeta^9
\end{pmatrix},
\qquad
\begin{pmatrix}
0&1&0\\
0&0&1\\
1&0&0
\end{pmatrix}.
\]

By \cite[Example~7.5]{CLY26}, among its six nontrivial
irreducible representations, three yield degree-zero sheaves 
supported on surfaces, while one yields a degree-zero sheaf 
supported on a curve. The remaining two representations have 
transforms concentrated in degree $-1$. Consequently,
\[
\mathcal{B}_0(G) = \frac{4}{6} = \frac{2}{3}.
\]
\end{ex}

Determining sharp bounds for $\mathcal{B}_0(G)$ in the
nonabelian setting requires further arguments:
the proof of Theorem~\ref{MainThm} uses the triangulation
and character labels available in the abelian case.

\bibliographystyle{alpha}

\end{document}